\documentclass{amsart}
\usepackage[margin=1in]{geometry}
\usepackage{amsmath}
\usepackage{amssymb}
\usepackage{amsfonts}
\usepackage{amsthm}
\usepackage{soul}
\usepackage{tikz}
\usepackage[noadjust]{cite}
\usepackage{enumerate}

\theoremstyle{plain}
\newtheorem{theorem}{Theorem}[section]
\newtheorem{corollary}[theorem]{Corollary}
\newtheorem{lemma}[theorem]{Lemma}
\newtheorem{proposition}[theorem]{Proposition}

\theoremstyle{definition}
\newtheorem{definition}[theorem]{Definition}
\newtheorem{example}[theorem]{Example}
\newtheorem{question}[theorem]{Question}

\theoremstyle{remark}

\newcommand{\N}{\mathbb{N}}
\newcommand{\Z}{\mathbb{Z}}
\renewcommand{\H}{\mathcal{H}}
\newcommand{\xx}{\mathbf{x}}
\newcommand{\yy}{\mathbf{y}}

\newcommand{\e}{\epsilon}
\newcommand{\se}{\subseteq}
\DeclareMathOperator{\orb}{Orb}

\newcommand{\forb}{\overrightarrow{Orb}}
\newcommand{\borb}{\overleftarrow{Orb}}

\title[Entropy of set-valued functions]{Topological Entropy of set-valued functions}

\author[J. P. Kelly]{James P. Kelly}
\address[J. P. Kelly]{Department of Mathematics, Christopher Newport University, Newport News, VA 23606--3072, USA}
\email{james.kelly@cnu.edu}

\author[T. Tennant]{Tim Tennant}
\address[T. Tennant]{Department of Mathematics, Baylor University, Waco, TX 76798--7328, USA}
\email{timothy\_tennant@baylor.edu}


\begin{document}
\begin{abstract}
Topological entropy is a widely studied indicator of chaos in topological dynamics. Here we give a generalized definition of topological entropy which may be applied to set-valued functions. We demonstrate that some of the well-known results concerning topological entropy of continuous (single-valued) functions extend naturally to set-valued functions while others must be altered. We also present sufficient conditions for a set-valued function to have positive or infinite topological entropy.
\end{abstract}
\maketitle

\section{Introduction}\label{Section Introduction}

The subject of topological entropy was first introduced by Adler,  Konheim, and McAndrew in 1965, and, in 1970, Bowen presented an equivalent definition in the context of metric spaces, \cite{1965,Bowen-Topological_entropy_and_axiom_a}. Topological entropy is a measure of the complexity of the dynamics of a function, and a function which has positive topological entropy is sometimes referred to as chaotic.

The study of topological entropy includes a variety of topics including sufficient conditions for a function to have positive or infinite topological entropy, the relationship between the topological entropy of a function and the structure of its inverse limit space, and what types of spaces admit positive entropy homeomorphisms, \cite{Mouron-Positive_entropy_homeomorphisms_of_chainable_continua,Ye-dynamics_of_homeomorphisms_of_hereditarily_decomposable_chainable_continua,Barge_Diamond-Dynamics_of_maps_finite_graphs,Misiurewicz-Horshoes_for_mappings_of_the_interval}. 

For many years, there has been an overlap between the study of the dynamics of a system and the study of the topological structure of its inverse limit. Some notable results in this area can be found in \cite{Barge_Diamond-Dynamics_of_maps_finite_graphs,Barge_Martin-Chaos_periodicity_and_snakelike_continua,Barge_Martin-The_construction_of_global_attractors}. In 2004, Mahavier began the study of inverse limits of upper semi-continuous, set-valued functions, \cite{Mahavier-Invlims_paper}. In recent years, there has been significant research in this area, primarily focusing on the continuum theoretic properties of these inverse limits. Many of the fundamental results concerning inverse limits of set-valued functions can be found in \cite{Ingram-Book_Setvalued_invlims}.

In this paper, we focus on the dynamics of upper semi-continuous, set-valued functions. We provide a generalization of Bowen's definition of topological entropy which may be applied to set-valued functions, and we demonstrate that some well-known results extend naturally to the more general setting while others do not.

In Section~\ref{Section Preliminary Definitions} we give some background definitions and present a definition for topological entropy of a set-valued function. We then begin our discussion of the topic by exploring some properties of topological entropy which generalize naturally to set-valued functions. We then show, in Section~\ref{Section Entropy of the shift map} that the topological entropy of a set-valued function is equal to the topological entropy of the shift map on its orbit spaces. (The orbit spaces are analogous to inverse limit spaces and are defined in Section~\ref{Section Preliminary Definitions}.) We also show that there is no loss of generality in assuming that the set-valued functions are surjective. In Section~\ref{Section Topological Conjugacy} we extend the notions of topological conjugacy and semi-conjugacy to set-valued functions and show that the results concerning these properties also generalize naturally to set-valued functions.

Next, we discuss some of the ways in which results concerning topological entropy of set-valued functions differ from the results in the traditional setting. In Section~\ref{Section Entropy of F^k}, we demonstrate the relationship between the topological entropy of a set-valued function and that of its iterates. Finally, we present sufficient conditions for a set-valued function to have positive topological entropy in Section~\ref{Section positive entropy} and sufficient conditions for infinite topological entropy in Section~\ref{Section infinite entropy}.

\section{Preliminary Definitions}\label{Section Preliminary Definitions}
Given a compact metric space $X$, we denote by $2^X$ the set of all non-empty compact subsets of $X$.

If $X$ and $Y$ are compact metric spaces, a function $F:X\rightarrow 2^Y$ is said to be \emph{upper semi-continuous} at a point $x\in X$ if, for every open set $V\se Y$ containing $F(x)$, there exists an open set $U\se X$ containing $x$ such that $F(t)\se V$ for all $t\in U$. $F$ is said to be \emph{upper semi-continuous} if it is upper semi-continuous at each point of $X$.

The \emph{graph} of a function $F:X\rightarrow2^Y$ is defined to be the set \[\Gamma(F)=\left\{(x,y)\in X\times Y:y\in F(x)\right\}.\] Ingram and Mahavier show, in \cite{Ingram_Mahavier-Invlims_paper}, that if $X$ and $Y$ are compact Hausdorff spaces, then $F:X\rightarrow2^Y$ is upper semi-continuous if, and only if, $\Gamma(F)$ is closed in $X\times Y$. If $f:X\rightarrow Y$, we may think of $f$ as a set-valued function by defining a function $\tilde{f}:X\rightarrow2^Y$ by $\tilde{f}(x)=\{f(x)\}$. In this case, $\tilde{f}$ is upper semi-continuous if and only if $f$ is continuous. For increased distinction, we will refer to an upper semi-continuous function $F:X\rightarrow2^Y$ as a \emph{set-valued function} and a continuous function $f:X\rightarrow Y$ as a \emph{mapping}. 

If $X,Y$, and $Z$ are compact metric spaces, $F:X\rightarrow2^Y$ and $G:Y\rightarrow2^Z$, we define $G\circ F:X\rightarrow2^Z$ by \[G\circ F(x)=\bigcup_{y\in F(x)}G(y).\] If $F$ and $G$ are upper semi-continuous, then $G\circ F$ is as well.

In this paper, we will be focusing on the setting where $X$ is a compact metric space and $F:X\rightarrow2^X$ is upper semi-continuous. In this case, the pair $(X,F)$ is called a \emph{topological dynamical system}. We define $F^0$ to be the identity on $X$, and for each $n\in\N$, we let $F^n=F\circ F^{n-1}$.

%We will develop a definition for topological entropy which generalizes the well-studied definition in the context of continuous, single-valued functions. %We say that $F$ is \emph{surjective} if for every $y\in X$, there exists $x\in X$ such that $y\in F(x)$.

%For the remainder of this paper, we will alway suppose that $X$ is a compact metric space, and $F:X\rightarrow2^X$ is upper semi-continuous. The pair $(X,F)$ will be referred to as a \emph{dynamical system}.

We begin the process of defining topological entropy for set-valued functions by defining multiple types of orbits for the system $(X,F)$. A \emph{forward orbit} for the system is a sequence $(x_0,x_1,x_2,\ldots)$ in $X$ such that for each $i\geq0$, $x_{i+1}\in F(x_i)$. A \emph{backward orbit} is a sequence $(\ldots,x_{-2},x_{-1},x_0)$ in $X$ such that for each $i\leq-1$, $x_{i+1}\in F(x_i)$. A \emph{full orbit} is a sequence $(\ldots,x_{-2},x_{-1},x_0,x_1,x_2,\ldots)$ in $X$ such that for each $i\in\Z$, $x_{i+1}\in F(x_i)$.

Finally, we will also consider finite orbits. Given a natural number $n$, an \emph{$n$-orbit} for the system $(X,F)$ is a finite sequence $(x_0,\ldots,x_{n-1})$ in $X$ such that for each $i=0,\ldots,n-2$, $x_{i+1}\in F(x_i)$.

A full orbit $\xx$ is called \emph{periodic} if there exists $m\in\N$ such that $x_i=x_{i+m}$ for all $i\in\Z$. If $\xx$ is periodic, the \emph{period} of $\xx$ is the smallest number $m\in\N$ for which $x_i=x_{i+m}$ for all $i\in\Z$.

\begin{definition}\label{Definition orbit spaces}
Given a set $A\se X$, and $n\in\N$, we define the following orbit spaces:
\begin{eqnarray}
\orb_n(A,F) &=& \{n\text{-orbits }(x_0,\ldots,x_{n-1}):x_0\in A\}\nonumber\\
\forb(A,F) &=& \{\text{forward orbits }(x_0,x_1,\ldots):x_0\in A\}\nonumber\\
\borb(A,F) &=& \{\text{backward orbits }(\ldots,x_{-1},x_0):x_0\in A\}\nonumber\\
\orb(A,F) &=& \{\text{full orbits }(\ldots,x_{-1},x_0,x_1,\ldots):x_0\in A\}\nonumber
\end{eqnarray}
\end{definition}

Each of these is given the subspace topology inherited as a subset of the respective product space. Let $d$ be the metric on $X$, and suppose that the diameter of $X$ is equal to 1. For each $n\in\N$, we define a metric $D$ on $\prod_{i=1}^nX$ by \[D(\xx,\yy)=\max_{0\leq i\leq n-1}d\left(x_i,y_i\right).\]

If $\mathbb{A}\in\{\Z,\Z_{\geq0},\Z_{\leq0}\}$ then we define a metric $\rho$ on $\prod_{i\in\mathbb{A}}X$ by \[\rho(\xx,\yy)=\sup_{i\in\mathbb{A}}\frac{d\left(x_i,y_i\right)}{|i|+1}.\] Also, for any set $L\se\mathbb{A}$, we define the \emph{projection map} $\pi_L:\prod_{i\in\mathbb{A}}X\rightarrow\prod_{i\in L}X$ by $\pi_L(\xx)=(x_i)_{i\in L}$.

In the past decade, there has been a significant amount of research concerning the inverse limits of upper semi-continuous set-valued functions. As it is typically defined, the inverse limit of the system $(X,F)$ indexed by $\Z_{\geq0}$ is equal to $\borb(X,F)$, and the inverse limit of the system indexed by $\Z$ is equal to $\orb(X,F)$. Also, $\forb(X,F)$ would be equal to the inverse limit of the system $(X,F^{-1})$ where $F^{-1}:X\rightarrow2^X$ is defined by $x\in F^{-1}(y)$ if, and only if, $y\in F(x)$. (Note that for $F^{-1}$ to be well-defined, it is assumed that $F$ is surjective, in the sense that for all $y\in X$, there exists $x\in X$ such that $y\in F(x)$.)

In the case where $f$ is a mapping, there is less need for this distinction between the various orbit spaces. In that case, $\borb(X,f)$ is homeomorphic to $\orb(X,f)$, and, for each $n\in\N$, $\orb_n(X,f)$ is homeomorphic to $X$.

We now begin our definition of topological entropy.  For the sake of completeness, we first give the definition in terms of a mapping before generalizing to set-valued functions.

\begin{definition}\label{Definition (n,e)-separated}
Let $X$ be a compact metric space. A set $S\se X$ is called \emph{$\e$-separated} if for each $x,y\in S$, $x\neq y$, $d(x,y)\geq\e$.
Let $f:X \rightarrow X$ be a mapping, and let $n\in \N$.  We say $S\se X$ is \emph{$(n,\e)$-separated} if for $x,y\in S$ with $x\neq y$, we have that
\[\max_{0\leq i \leq n-1}d\left(f^i(x),f^i(y)\right) \geq \e.\]
We denote by $s_{n,\e}(f)$ the largest cardinality of an $(n,\e)$-separated set with respect to $f$.  When there is no ambiguity, we shall use $s_{n,\e}$.
\end{definition}

\begin{definition}\label{Definition entropy}
Given $\e>0$, the \emph{$\e$-entropy} of $f$ is defined to be \[h(f,\e)=\limsup_{n\rightarrow\infty}\frac{1}{n}\log s_{n,\e},\]
and the \emph{topological entropy} of $f$ is defined to be \[h(f)=\lim_{\e\rightarrow0}h(f,\e).\]
\end{definition}

To adapt this definition to the context of set-valued functions, we work in $\orb_n(X,F)$ with the metric defined above, to preserve the idea of ``separated" meaning separated in at least one coordinate.

\begin{definition}\label{set-valued (n,e)-separated}
Let $(X,F)$ be a topological dynamical system, and let $n\in\N$ and $\e>0$. An \emph{$(n,\e)$-separated} set for $F$ is an $\e$-separated subset of $\orb_n(X,F)$. We denote by $s_{n,\e}(F)$, the largest cardinality of an $(n,\e)$-separated set with respect to $F$. When no ambiguity shall arise, we simply write $s_{n,\e}$.
\end{definition}

\begin{definition}\label{set-valued Definition entropy}
Given $\e>0$, the \emph{$\e$-entropy} of $F$ is defined to be \[h(F,\e)=\limsup_{n\rightarrow\infty}\frac{1}{n}\log s_{n,\e},\]
and the \emph{topological entropy} of $F$ is defined to be \[h(F)=\lim_{\e\rightarrow0}h(F,\e).\]
\end{definition}

Just as in the case of a mapping on $X$, we may give an equivalent definition using spanning sets rather than separated sets. 

\begin{definition}\label{set-valued Definition (n,e)-spanning}
Let $X$ be a compact metric space. A set $S\se X$ is called \emph{$\e$-spanning} if for each $y\in X$, there exists $x\in S$ with $d(x,y)<\e$.

Let $(X,F)$ be a topological dynamical system, and let $n\in\N$ and $\e>0$. An \emph{$(n,\e)$-spanning} set for $F$ is an $\e$-spanning subset of $\orb_n(X,F)$. We denote by $r_{n,\e}(F)$, the smallest cardinality of an $(n,\e)$-spanning set with respect to $F$.
\end{definition}

It is shown in \cite{Tennant-Specification} that \[\lim_{\e\rightarrow0}\limsup_{n\rightarrow\infty}\frac{1}{n}\log r_{n,\e}(F)=\lim_{\e\rightarrow0}\limsup_{n\rightarrow\infty}\frac{1}{n}\log s_{n,\e}(F).\] Thus, either notion may be used to define the topological entropy of $F$.

%Finally, given a set $S$, we denote by $|S|$ the cardinality of $S$.

\section{Topological Entropy of the Shift Map on an Orbit Space}\label{Section Entropy of the shift map}

In \cite{Bowen-Topological_entropy_and_axiom_a}, Bowen shows that the entropy of a mapping on $X$ is equal to the entropy of the shift map on the inverse limit space. In this section, we establish analogous results by showing that the entropy of $F$ is equal to the entropy of the shift maps on any of the orbit spaces defined in Definition~\ref{Definition orbit spaces}.

\begin{theorem}\label{Theorem shift map forward orbits}
Let $(X,F)$ be a topological dynamical system. If $\sigma:\forb(X,F)\rightarrow\forb(X,F)$ is the shift map defined by \[\sigma\left(x_0,x_1,x_2,\ldots\right)=\left(x_1,x_2,x_3,\ldots\right),\] then $h(\sigma)=h(F)$.
\end{theorem}
\begin{proof}
Let $n\in\N$ and $\e>0$. We will show that $s_{n,\e}(F)\leq s_{n,\e}(\sigma)$. Let $S\se\orb_n(X,F)$ be an $(n,\e)$-separated set for $F$ of maximal cardinality. Each $n$-orbit $(x_0,\ldots,x_{n-1})\in S$ may be extended to an infinite forward orbit in $\forb(X,F)$. Let $T\se\forb(X,F)$ be the set of all such forward orbits.

Claim: $T$ is an $(n,\e)$-separated set for $\sigma$ as defined in Definition~\ref{Definition (n,e)-separated}.

To see this, let $\xx,\yy\in T$. Then $(x_0,\ldots,x_{n-1})$ and $(y_0,\ldots,y_{n-1})$ are in $S$, so $d(x_j,y_j)\geq\e$ for some $0\leq j\leq n-1$. Thus, \[\rho\left(\sigma^j(\xx),\sigma^j(\yy)\right)=\sup_{i\geq0}\frac{d\left(x_{i+j},y_{i+j}\right)}{i+1}\geq d\left(x_j,y_j\right)\geq\e.\]

Thus we have that $s_{n,\e}(F)\leq s_{n,\e}(\sigma)$ for all $n\in\N$ and $\e>0$. If follows that $h(F)\leq h(\sigma)$.

Next, fix $\e>0$, and choose $k\in\N$ with $1/k<\e$. We show that for each $n\in\N$, $s_{n+k,\e}(\sigma)\leq s_{n,\e}(F)$. Let $S\se\forb(X,F)$ be an $(n,\e)$-separated set for $\sigma$ of maximal cardinality (as defined in Definition~\ref{Definition (n,e)-separated}). Then, for each $\xx,\yy\in S$, there exists $j=0,\ldots,n-1$ such that $\rho(\sigma^j(\xx),\sigma^j(\yy))\geq\e$. Thus, there exists $i\in\N$ such that \[\e\leq\frac{d\left(x_{i+j},y_{i+j}\right)}{i+1}\leq d\left(x_{i+j},y_{i+j}\right).\] Since $1/k<\e$, it follows that $i+1<k$. Thus we have that $i < k$ and $j\leq n-1$, so $i+j< n+k-1$.

Therefore, if $T=\{(x_0,\ldots,x_{n+k-1}):\xx\in S\}$, then $T$ is an $(n+k,\e)$-separated set for $F$. Moreover, \[s_{n,\e}(\sigma)=|S|=|T|\leq s_{n+k,\e}(F),\] and it follows that $h(\sigma)\leq h(F)$.
\end{proof}

In order to establish similar results for the shift maps on $\borb(X,F)$ and $\orb(X,F)$, we must first establish that there is no loss of generality in assuming that $F$ is surjective. Bowen established this fact for mappings in \cite{Bowen-Topological_entropy_and_axiom_a}.

\begin{definition}\label{Definition non-wandering}
Let $X$ be a compact metric space, and $f:X\rightarrow X$ be a mapping. A point $x\in X$ is called \emph{non-wandering} if for every open set $U\se X$ containing $x$, there exists $n\in\N$ such that $f^n(U)\cap U\neq\emptyset$.
\end{definition}

\begin{theorem}[Bowen]\label{Theorem Bowen non-wandering}
Let $X$ be a compact metric space, and $f:X\rightarrow X$ be a mapping. If $\Omega$ is the set of non-wandering points then $h(f)=h(f|_{\Omega})$.
\end{theorem}

Note that if $C=\bigcap_{n\in\N}f^n(X)$, then $C$ contains all the non-wandering points, so it follows from Theorem~\ref{Theorem Bowen non-wandering} that the entropy of $f$ is equal to the entropy of $f|_C$. We show in the following lemma that the same holds for upper semi-continuous set-valued functions.

\begin{lemma}\label{Lemma surjective core}
Let $(X,F)$ be a topological dynamical system, and let $C=\bigcap_{n\in\N}F^n(X)$. Then $h(F)=h(F|_C)$.
\end{lemma}
\begin{proof}
First, note that $F(C)=C$. Also, since $C=\bigcap_{n\in\N}F^n(X)$, it follows that \[\forb\left(C,F|_C\right)=\bigcap_{n\in\N}\sigma^n\left(\forb(X,F)\right).\] Let $\widetilde{C}=\forb(C,F|_C)$. Since $\sigma$ is a mapping, we have from Theorem~\ref{Theorem Bowen non-wandering} that $h(\sigma)=h(\sigma|_{\widetilde{C}})$. Then, by Theorem~\ref{Theorem shift map forward orbits}, we have that $h(F)=h(\sigma)$, and $h(F|_C)=h(\sigma|_{\widetilde{C}})$. The result follows.
\end{proof}

\begin{theorem}\label{Theorem shift map other orbit spaces}~

\begin{enumerate}
%\item\label{Theorem shift map 1} If $\sigma:\forb(X,F)\rightarrow\forb(X,F)$ is the shift map defined by \[\sigma(x_0,x_1,x_2,\ldots)=(x_1,x_2,\ldots),\] then $h(\sigma)=h(F)$.
\item\label{Theorem shift map 2} If $\sigma:\borb(X,F)\rightarrow\borb(X,F)$ is the shift map defined by \[\sigma\left(\ldots,x_{-2},x_{-1},x_0\right)=\left(\ldots,x_{-3},x_{-2},x_{-1}\right)\] then $h(\sigma)=h(F)$.
\item\label{Theorem shift map 3} If $\sigma:\orb(X,F)\rightarrow\orb(X,F)$ is the shift map defined by $\sigma(\xx)=\yy$ where for each $i\in\Z$, $y_i=x_{i+1}$, then $h(\sigma)=h(F)$.
\end{enumerate}
\end{theorem}
\begin{proof}
For either shift map, $\sigma$, the same argument as in the proof of Theorem~\ref{Theorem shift map forward orbits} may be used to show that $h(\sigma)\leq h(F)$. Then by Lemma~\ref{Lemma surjective core}, we may suppose without loss of generality that $F$ is surjective. Thus, each $n$-orbit for $F$ may be extended to an infinite backward (or full) orbit, so the argument used in Theorem~\ref{Theorem shift map forward orbits} may be used to show that $h(F)\leq h(\sigma)$.
\end{proof}

\begin{corollary}\label{Corollary F inverse}
Let $(X,F)$ be a topological dynamical system with $F$ surjective. Then $h(F)=h(F^{-1})$.
\end{corollary}

Theorem~\ref{Theorem shift map forward orbits} and Theorem~\ref{Theorem shift map other orbit spaces} are significant for multiple reasons. First, all of the shift maps considered are mappings, and the shift on $\orb(X,F)$ is a homeomorphism. Thus, the large volume of research on the topic of topological entropy of mappings and homeomorphisms may be applied to study the entropy of set-valued functions.

Second, there are multiple ways in which topological entropy may be defined which, in the context of mappings, are all equivalent. Theorem~\ref{Theorem shift map forward orbits} and Theorem~\ref{Theorem shift map other orbit spaces} show that any definition of topological entropy for set-valued functions which generalizes one of the definitions for topological entropy of mappings is equivalent to Definition~\ref{set-valued Definition entropy} so long as a theorem such as Theorem~\ref{Theorem shift map forward orbits} or Theorem~\ref{Theorem shift map other orbit spaces} holds for that definition.

\section{Topological Conjugacy and Semi-Conjugacy}\label{Section Topological Conjugacy}

Another concept regarding topological entropy which generalizes nicely to the context of set-valued functions is the notion of topological conjugacy and semi-conjugacy.

\begin{definition}\label{Definition topological conjugacy}
Let $(X,F)$ and $(Y,G)$ be topological dynamical systems. We say that $G$ is \emph{topologically semi-conjugate} to $F$ if there exists a continuous surjection $\varphi:X\rightarrow Y$ such that for all $x\in X$, \[G\circ\varphi(x)\se\varphi\circ F(X).\]
The surjection $\varphi$  is called a \emph{topological semi-conjugacy} from $(X,F)$ to $(Y,G)$.

We say that $F$ and $G$ are \emph{topologically conjugate} if there exists a homeomorphism $\varphi:X\rightarrow Y$ such that $G\circ\varphi=\varphi\circ F$. The homeomorphism $\varphi$ is called a \emph{topological conjugacy} between $(X,F)$ and $(Y,G)$.
\end{definition}

The following theorems generalize well-known results regarding the topological entropy of topologically conjugate or semi-conjugate mappings (see \cite[Theorem~7.2]{Walters-Book_ergodic_theory})

\begin{theorem}\label{Theorem semi-conjugacy}
Let $(X,F)$ and $(Y,G)$ be topological dynamical systems. If $G$ is topologically semi-conjugate to $F$, then $h(G)\leq h(F)$.
\end{theorem}
\begin{proof}
Let $\varphi:X\rightarrow Y$ be a topological semi-conjugacy from $(X,F)$ to $(Y,G)$. Let $\e>0$, and choose $\delta>0$ so that if $a,b\in X$ with $d(a,b)<\delta$, then $d(\varphi(a),\varphi(b))<\e/2$. For each $n\in\N$, define $\Phi_n:\orb_n(X,F)\rightarrow Y^n$ by \[\Phi_n\left(x_0,\ldots,x_{n-1}\right)=\left(\varphi\left(x_0\right),\ldots,\varphi\left(x_{n-1}\right)\right).\]

We show that for each $n\in\N$, $\orb_n(Y,G)\se\Phi_n[\orb_n(X,F)]$. Let $\yy\in\orb_n(Y,G)$. Choose any $x_0\in\varphi^{-1}(y_0)$. Now suppose that $x_i\in\varphi^{-1}(y_i)$ has been chosen for some $0\leq i\leq n-2$ such that $(x_0,x_1,\ldots, x_{n-2})\in \orb_{n-1}(X,F)$. Since \[y_{i+1}\in G\left(y_i\right)=G\circ\varphi\left(x_i\right)\se\varphi\circ F\left(x_i\right),\] there exists $x_{i+1}\in F(x_i)$ such that $\varphi(x_{i+1})=y_{i+1}$. In this manner, we construct an $n$-orbit $\xx\in\orb_n(X,F)$ such that $\Phi_n(\xx)=\yy$.

Fix $n\in\N$, and let $S$ be an $(n,\delta)$-spanning set for $F$ of minimum cardinality. Let $T=\Phi_n(S)$. Then $T$ $\e/2$-spans $\orb_n(Y,G)$. To see this, let $\yy\in\orb_n(Y,G)$, and choose $\xx\in\Phi^{-1}_n(\yy)$. Since $S$ is an $(n,\delta)$-spanning set, there exists $\mathbf{s}\in S$ such that $D(\mathbf{s},\xx)<\delta$. Then $\Phi_n(\mathbf{s})\in T$, and it follows from the choice of $\delta$ that $D(\Phi_n(\mathbf{s}),\yy)<\e/2$.

Since $T$ is not necessarily a subset of $\orb_n(Y,G)$, it may not satisfy the definition of an $(n,\e/2)$-spanning set for $G$. However, we may use $T$ to construct an $(n,\e)$-spanning set for $G$. For each $\mathbf{t}\in T$, if the $D$-ball centered at $\mathbf{t}$ of radius $\e/2$ intersects $\orb_n(Y,G)$, then choose any $\mathbf{t}'$ in that intersection. Let $T'$ be the collection of all such points $\mathbf{t'}$, and note that $|T'|\leq|T|$. It follows from the triangle inequality that $T'$ is an $(n,\e)$-spanning set for $G$.

Therefore, for all $n\in\N$, \[r_{n,\delta}(F)=|S|\geq|T'|\geq r_{n,\e}(G).\] It follows that $h(F)\geq h(G)$.
\end{proof}

If two systems are topologically conjugate, then, in particular, each is topologically semi-conjugate to the other. Hence, the following theorem follows immediately from Theorem~\ref{Theorem semi-conjugacy}.

\begin{theorem}\label{Theorem conjugacy}
If $(X,F)$ and $(Y,G)$ are topologically conjugate dynamical systems, then $h(F)=h(G)$.
\end{theorem}

\section{Topological Entropy of Iterates of a Set-valued Function}\label{Section Entropy of F^k}

One result concerning topological entropy of mappings which does not always hold in the context of upper semi-continuous set-valued functions is the relationship of the entropy of a function to the entropy of its iterates. In the setting of mappings on compact metric spaces, we have the following well-known result (see \cite[Theorem 7.10]{Walters-Book_ergodic_theory} for a proof).

\begin{theorem}\label{Theorem h(f^k)=kh(f)}
Let $X$ be a compact metric space, and let $f:X\rightarrow X$ be continuous. Then for all $k\in\N$, $h(f^k)=kh(f)$.
\end{theorem}

This need not hold in general for upper semi-continuous set-valued functions. However, we show in Theorem~\ref{Theorem bounding entropy} that for any topological dynamical system $(X,F)$ and any $k\in\N$, $h(F)\leq h(F^k)\leq kh(F)$. We begin with the following lemma.

\begin{lemma}\label{Lemma combinatorics}
Let $(X,F)$ be a topological dynamical system, $n\in\N$, $\e>0$, and $S$ an $(n,\e)$-separated set for $F$. Let $k,m\in\N$, such that $(m-1)k<n\leq mk$, and let $L=n-(m-1)k$. 

For each $i=0,\ldots,L-1$, let \[A_i=\{i,i+k,i+2k,\ldots,i+(m-1)k\},\] and for each $i=L,\ldots,k-1$, let \[A_i=\{i,i+k,i+2k,\ldots,i+(m-2)k\}.\] If, for each $i=0,\ldots,k-1$, $S_i$ is chosen to be the largest $\e/2$-separated subset of $\pi_{A_i}(S)$, then
\begin{equation}
|S|\leq\prod_{i=0}^{k-1}|S_i|\label{inequality 1}\nonumber
\end{equation}
\end{lemma}
\begin{proof}
%First, it holds that
%\begin{equation}
%|S|\leq\prod_{i=0}^{k-1}\left|\pi_{A_i}(S)\right|.\label{inequality 2}
%\end{equation}
%Suppose that $0\leq j\leq k-1$, $\xx,\yy\in\pi_{A_j}(S)$ and for each $l=0,\ldots,k-1$, $l\neq j$, $\mathbf{z(l)}\in\pi_{A_l}(S)$. Then, if $\xx$ and $\yy$ are not $\e$-separated (i.e. $d(x_i,y_i)<\e$ for all $i\in A_j$), then $S$ may contain a point $\mathbf{p}$ such that $\pi_{A_j}(\mathbf{p})=\xx$, and $\pi_{A_l}(\mathbf{p})=\mathbf{z(l)}$ for all $l=0,\ldots,k-1$, $l\neq j$, or $S$ may contain a point $\mathbf{q}$ such that $\pi_{A_j}(\mathbf{q})=\yy$, and $\pi_{A_l}(\mathbf{q})=\mathbf{z(l)}$ for all $l=0,\ldots,k-1$, $l\neq j$. However, $S$ cannot contain both of these points or else it would not be $\e$-separated. In this sense, the inequality \eqref{inequality 2} represents an overestimation for the cardinality of $S$.

%To see that \eqref{inequality 1}, a (possibly) tighter inequality, holds, 
Define  $T\se X^n$ to be the set \[T=\bigcap_{i=0}^{k-1}\pi_{A_i}^{-1}(S_i).\] Then \[|T|=\prod_{i=0}^{k-1}\left|S_i\right|.\] Now, $T$ is not necessarily a subset of $S$ (or even of $\orb_n(X,F)$) nor is $S$ necessarily a subset of $T$. However, we will show that $|S|\leq|T|$ by demonstrating that $|S\setminus T|\leq|T\setminus S|$.

Suppose $\xx\in S\setminus T$. For each $j=0,\ldots,k-1$, consider the point $\pi_{A_j}(\xx)$, and define \[T_j(\xx)=\left\{\yy\in S_j:D\left(\yy,\pi_{A_j}(\xx)\right)<\frac{\e}{2}\right\}.\] Since $\xx$ is not in $T$, there is some $0\leq j\leq k-1$ such that $\pi_{A_j}(\xx)\notin S_j$, and hence $\pi_{A_j}(\xx)\notin T_j(\xx)$. However, since $S_j$ is the largest $\e/2$-separated subset of $\pi_{A_j}(S)$, it follows that $T_j(\xx)\neq\emptyset$ for each $0\leq j\leq k-1$. 

Now define \[T(\xx)=\bigcap_{i=1}^{k-1}\pi_{A_i}^{-1}\left[T_i(\xx)\right].\] Then for each $\mathbf{z}\in T(\mathbf{x})$, $D(\xx,\mathbf{z})<\e/2$. Hence, since $\xx\in S$, and $S$ is $\e$-separated, $\mathbf{z}\notin S$. Since this holds for all $\mathbf{z}\in T(\xx)$, we have that $T(\xx)\cap S=\emptyset$. Moreover, since for each $0\leq j\leq k-1$, $|T_j(\xx)|\geq1$, it follows that $|T(\xx)|\geq1$. Hence, for each point $\xx\in S\setminus T$, there is at least one point $\mathbf{z}\in T(\xx)\setminus S\se T\setminus S$.

Finally, if $\xx,\yy\in S\setminus T$, then $T(\xx)\cap T(\yy)=\emptyset$. This is because if there were a sequence $\mathbf{z}$ in $T(\xx)\cap T(\yy)$, then $D(\xx,\yy)\leq D(\xx,\mathbf{z})+D(\yy,\mathbf{z})<\e$ which would contradict $S$ being $\e$-separated. Therefore, we have that $|T\setminus S|\geq|S\setminus T|$, and the result follows.

%To see that \eqref{inequality 1}, a (possibly) tighter inequality, holds, for each $0\leq l\leq k-1$ and $\xx\in\pi_{A_l}(S)$ we define the sets $S(\xx)=\{\mathbf{p}\in S:\pi_{A_l}(\mathbf{p})=\xx\}$ and $S_l(\xx)=\{\yy\in S_l:0<D(\xx,\yy)<\e\}$. Observe that if $\mathbf{p}\in S(\xx)$, there cannot be any orbit $\mathbf{q}\in S$ with $q_i=p_i$ for $i\notin A_l$ and $\pi_{A_l}(\mathbf{q})\in S_l(\xx)$. Thus, as we consider how many elements the set $S$ may have, if $\xx\notin S_l$, we may say that \[|S(\xx)|\leq\sum_{\yy\in S_l(\xx)}|S(\yy)|.\]
%
%Thus, in counting the number of orbits in $S$, it suffices to consider only the elements of $S_l$ for each $l=0,\ldots,k-1$. Therefore, \eqref{inequality 1} holds.
\end{proof}

\begin{lemma}\label{Lemma bounding entropy}
Let $(X,F)$ be a topological dynamical system, and let $k\in\N$. Then for all $n\in\N$ and $\e>0$, if $m\in\N$ is chosen such that $(m-1)k<n\leq mk$, then \[s_{n,\e}(F)\leq\left[s_{m,\e/2}\left(F^k\right)\right]^k.\]
\end{lemma}
\begin{proof}
Let $n\in\N$ and $\e>0$, and fix $m\in\N$ such that $(m-1)k<n\leq mk$. Let $S$ be an $(n,\e)$-separated set for $F$ of maximal cardinality, and let $L=n-(m-1)k$. For each $i=0,\ldots,L-1$, let $A_i=\{i,i+k,i+2k,\ldots,i+(m-1)k\}$, and for each $i=L,\ldots,k-1$, let $A_i=\{i,i+k,i+2k,\ldots,i+(m-2)k\}$.

For each $i=1,\ldots,k-1$ choose $S_i$ to be the largest $\e/2$-separated subset of $\pi_{A_i}(S)$. By Lemma~\ref{Lemma combinatorics}, \[|S|\leq\prod_{i=0}^{k-1}|S_i|.\]

Moreover, for $i=0,\ldots,L-1$, $S_i$ is an $(m,\e/2)$-separated set for $F^k$, and for $i=L,\ldots,k-1$, $S_i$ is an $(m-1,\e/2)$-separated set for $F^k$. In either case, we have that $|S_i|\leq s_{m,\e/2}(F^k)$. Therefore \[s_{n,\e}(F)=|S|\leq\prod_{i=0}^{k-1}|S_i|\leq\left[s_{m,\e/2}\left(F^k\right)\right]^k.\]
\end{proof}

\begin{theorem}\label{Theorem bounding entropy}
Let $(X,F)$ be a topological dynamical system, and let $k\in\N$. Then \[h(F)\leq h\left(F^k\right)\leq kh(F).\]
\end{theorem}
\begin{proof}
First, to show that $h(F^k)\leq kh(F)$, let $n\in\N$, and let $S$ be an $(n,\e)$-separated set for $F^k$ of maximal cardinality. For each $(x_0,\ldots,x_{n-1})\in S$, choose $(y_0,\ldots,y_{nk-1})\in\orb_{nk}(F,X)$ such that for each $i=1,\ldots,n-1$, $y_{ik}=x_i$, and let $\widetilde{S}$ be the set of all such $nk$-orbits for $F$.

Then $\widetilde{S}$ is an $(nk,\e)$-separated set for $F$ with the same cardinality as $S$ but not necessarily of maximal cardinality. It follows that \[
s_{n,\e}\left(F^k\right)\leq s_{nk,\e}(F)\] and hence
\[\limsup_{n\rightarrow\infty}\frac{1}{n}\log s_{n,\e}\left(F^k\right) \leq k\limsup_{n\rightarrow\infty}\frac{1}{nk}\log s_{nk,\e}(F).\] Therefore $h(F^k)\leq kh(F)$.

To show the other inequality, note that from Lemma~\ref{Lemma bounding entropy}, if $n\in\N$, and $m\in\N$ is chosen so that $(m-1)k<n\leq mk$, then \[s_{n,\e}(F)\leq\left[s_{m,\e/2}\left(F^k\right)\right]^k.\] In this construction, $m\rightarrow\infty$ as $n\rightarrow\infty$, so \[\limsup_{n\rightarrow\infty}\frac{1}{n}\log s_{n,\e}(F)\leq\limsup_{m\rightarrow\infty}\frac{1}{n}\log\left[s_{m,\e/2}\left(F^k\right)\right]^k=\limsup_{m\rightarrow\infty}\frac{\alpha}{m}\log s_{m,\e/2}\left(F^k\right)\] where $\alpha=mk/n$. 

It follows from the inequality, $(m-1)k<n\leq mk$ that $\alpha\rightarrow 1$ as $n\rightarrow\infty$. Hence, we have that $h(F)\leq h(F^k)$.
\end{proof}

\begin{corollary}\label{Corollary iterates}
Let $(X,F)$ be a topological dynamical system, and let $k\in\N$. Then the following hold.
\begin{enumerate}
\item $h(F)=0$ if, and only if, $h(F^k)=0$.
\item $h(F)=\infty$ if, and only if, $h(F^k)=\infty$.
\item $0<h(F)<\infty$ if, and only if, $0<h(F^k)<\infty$.
\end{enumerate}
\end{corollary}

The inequality $h(F)\leq h(F^k)\leq kh(F)$ is most interesting when the entropy of $F$ is positive and finite. From Theorem~\ref{Theorem h(f^k)=kh(f)}, we have that for any mapping $f$, $h(f^k)=kh(f)$ for all $k\in\N$. Next, we give an example of two set-valued functions on the two element set $\{0,1\}$: one where $h(F^2)=h(F)$, and one where $h(F)<h(F^2)<2h(F)$.

\begin{example}\label{Example bounded entropy}
Let $X=\{0,1\}$.
\begin{enumerate}
\item Let $F:X\rightarrow2^X$ be defined by $F(0)=\{1\}$, and $F(1)=\{0,1\}$.  Then $h(F)=\log\varphi$, where $\varphi=(1+\sqrt{5})/2$, and $h(F^2)=\log 2$.

\item  Let $G:X\rightarrow2^X$ be defined by $G(0)=G(1)=\{0,1\}$. Then for all $k\in\N$, $h(G^k)=h(G)=\log 2$.
\end{enumerate}
\end{example}
\begin{proof}
Note that if $0<\e<1$, then for all $n\in\N$, the entire space of $n$-orbits is an $(n,\e)$-separated set (for $F$ and $G$ respectively).

For $F$, the sequence $(s_{n,\e})_{n=1}^\infty$ is a Fibonacci sequence beginning with $(2,3)$. Thus, $s_{n,\e}\approx5^{-1/2}\varphi^{n+2}$, and we have that $h(F)=\log\varphi$.

Now $F^2(0)=F^2(1)=\{0,1\}$, so $\orb(X,F^2)=\{0,1\}^\Z$, and the entropy of the shift on this space is $\log 2$. Thus $h(F^2)=\log2$ which is strictly between $h(F)$ and $2h(F)$.

Note that $G=F^2$, so we have that $h(G)=\log2$. Also, for any $k\in\N$, $G^k=G$, so, in particular, $h(G^k)=h(G)$.
\end{proof}

In this example, we had that $G^k=G$ for all $k\in\N$. This is not necessary, however, for their entropies to be equal. In the following example we present a function $F:[0,1]\rightarrow2^{[0,1]}$ for which $F^2\neq F$ but $h(F^2)=h(F)$. (The inverse limits of $F$ and $F^2$ are discussed in \cite[Example 4]{Ingram_Mahavier-Invlims_paper}.)

\begin{example}\label{Example h(F)=h(F^2) but F neq F^2}
Let $I=[0,1]$, and let $F:I\rightarrow2^I$ be defined by 
\[
F(x)=\begin{cases}
\left\{x+\frac{1}{2},\frac{1}{2}-x\right\} & x\leq\frac{1}{2}\\
\left\{x-\frac{1}{2},\frac{3}{2}-x\right\} & x\geq\frac{1}{2}
\end{cases}
\]
Then, $F^2\neq F$, but $h(F^2)=h(F)=\log2$. (The graphs of $F$ and $F^2$ are pictured in Figure~\ref{Figure diamond and cross}.)
\end{example}
\begin{proof}
For each $0<\e<1/4$, let $A_\e$ be the largest $\e$-separated subset of the set \[\left[0+\frac{\e}{2},\frac{1}{2}-\frac{\e}{2}\right]\cup\left[\frac{1}{2}+\frac{\e}{2},1-\frac{\e}{2}\right].\] Note that the cardinality of $A_\e$ is no more than three less than the largest cardinality for an $\e$-separated subset of $I$. 

Moreover, for each $a\in A_\e$, $F(a)$ contains exactly two points, and those points are at least $\e$ apart from each other. It follows that for each $n\in\N$, \[|A_\e|2^n\leq s_{n,\e}(F)\leq\left(|A_\e|+3\right)2^n,\] and thus, $h(F)=\log2$.

A similar argument shows that that $h(F^2)=\log2$.
\end{proof}

\begin{figure}
\begin{minipage}{.49\textwidth}
\centering
\begin{tikzpicture}[scale=5]
\draw[dotted] (0,0) node[left]{0} -- (0,1) node[left]{1} -- (1,1) -- (1,0) node[below]{1} -- (0,0) node[below]{0};
\draw[very thick, blue, join=bevel] (0,1/2)--(1/2,1)--(1,1/2)--(1/2,0)--(0,1/2);
\end{tikzpicture}
\end{minipage}
\begin{minipage}{.49\textwidth}
\centering
\begin{tikzpicture}[scale=5]
\draw[dotted] (0,0) node[left]{0} -- (0,1) node[left]{1} -- (1,1) -- (1,0) node[below]{1} -- (0,0) node[below]{0};
\draw[very thick, blue] (0,0)--(1,1);
\draw[very thick, blue] (0,1)--(1,0);
\end{tikzpicture}
\end{minipage}
\caption{Set-valued function $F$ (left) and $F^2$ (right) from Example~\ref{Example h(F)=h(F^2) but F neq F^2}}\label{Figure diamond and cross}
\end{figure}
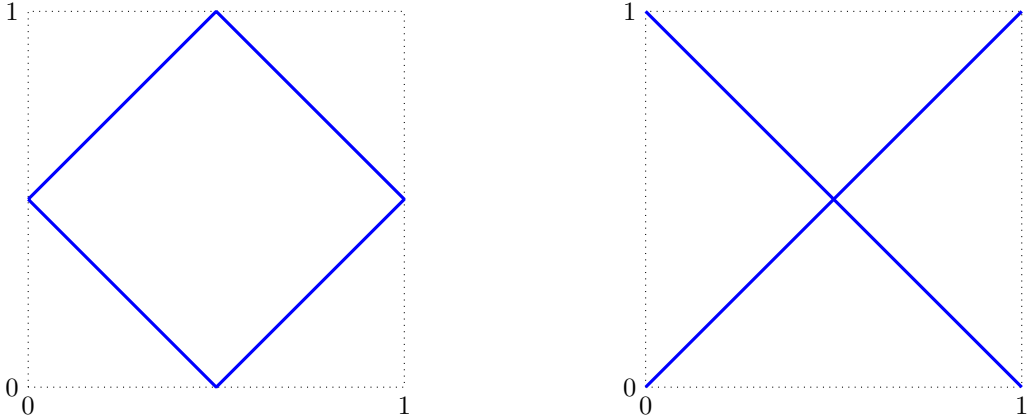

\section{Positive Topological Entropy}\label{Section positive entropy}

Each of the examples from Section~\ref{Section Entropy of F^k} illustrates functions with positive topological entropy, where the positive entropy may be witnessed on any compact subset. An interesting question is to determine ``minimal'' conditions for a set-valued function to have positive entropy. In this section, we establish conditions which are sufficient for a set-valued function to have positive entropy, and we demonstrate that set-valued functions satisfying these conditions may exhibit seemingly minimal chaotic behavior. %In particular, in Example~\ref{Example positive entropy on a nowhere dense set} we present a topological dynamical system on $[0,1]$ which has positive topological entropy, but the positive entropy only occurs on a finite subset of $[0,1]$.

We also discuss the relationship between periodicity and positive topological entropy. A mapping on $[0,1]$ has positive topological entropy if, and only if, it has a periodic point whose period is not a power of 2. We demonstrate that this equivalence does not hold for set-valued functions on the interval.

We begin with sufficient conditions for a set-valued function to have positive topological entropy.

\begin{proposition}\label{Proposition positive entropy box}
Let $(X,F)$ be a topological dynamical system.  Let $a,b\in X$, with $a\neq b$. If $\{a,b\}\se F(a)$ and $\{a,b\}\se F(b)$, then $h(F)\geq\log2$.
\end{proposition}
\begin{proof}
For each $n\in\N$ and each $0<\e<d(a,b)$, the set $\{a,b\}^n\se\orb_n(X,F)$ is an $(n,\e)$-separated set. Thus, $s_{n,\e}\geq2^n$. It follows that $h(F)\geq\log2$.
\end{proof}

Under the assumptions of Proposition~\ref{Proposition positive entropy box}, $a$ has two distinct periodic orbits, $(a,a,a,\ldots)$ and $(a,b,a,b,\ldots)$. The next theorem generalizes Proposition~\ref{Proposition positive entropy box} by focusing on this property.

In this theorem, given two finite sequences $\mathbf{u}=(u_i)_{i=0}^n$ and $\mathbf{v}=(v_i)_{i=0}^n$, we define $\mathbf{uv}$ to be the sequence $(u_0,\ldots,u_n,v_0,\ldots,v_n)$. We also define a \emph{finite word of length $m$} from $\{\mathbf{u},\mathbf{v}\}$ to be a sequence of the form $\mathbf{a_1a_2\cdots a_m}$ where for each $1\leq j\leq m$, $\mathbf{a_j}\in\{\mathbf{u},\mathbf{v}\}$.

\begin{theorem}\label{Theorem two periodic orbits}
Let $(X,F)$ be a topological dynamical system.  Suppose there exists a point $p\in X$ and two distinct periodic orbits $\mathbf{a}$ and $\mathbf{b}$ such that $a_0=b_0=p$.  Then $h(F)>0$.
\end{theorem}
\begin{proof}
Let $m$ be the period of $\mathbf{a}$, let $k$ be the period of $\mathbf{b}$, and let $l$ be the least common multiple of $m$ and $k$. Let $\mathbf{u}=(a_0,\ldots,a_{l-1})$, and let $\mathbf{v}=(b_0,\ldots,b_{l-1})$. Note that $p\in F(a_{l-1})$ and $p\in F(b_{l-1})$, so any finite word from $\{\mathbf{u},\mathbf{v}\}$ is a finite orbit for $F$. Also, since $\mathbf{a}$ and $\mathbf{b}$ are not equal, neither are $\mathbf{u}$ and $\mathbf{v}$, so there exists $0\leq j\leq l-1$ such that $u_j\neq v_j$.

For each $n\in\N$, let $S_n$ be the set of all finite words of length $n$ from $\{\mathbf{u},\mathbf{v}\}$. Then $S_n$ is an $nl$-orbit. Moreover, if $0<\e<d(u_j,v_j)$, then $S_n$ is an $(nl,\e)$-separated set, and $|S_n|=2^n$. It follows that $s_{nl,\e}\geq 2^n$, and hence, $h(F)\geq (\log2)/l>0$.
\end{proof}

\begin{example}\label{Example positive entropy on a nowhere dense set}
Let $I=[0,1]$, and let $F:I\rightarrow2^I$ be defined by $F(x)=\{x\}$ for $0<x<1$, and $F(0)=F(1)=\{0,1\}$ (pictured in Figure~\ref{Figure positive entropy on a nowhere dense set}). Then, according to Proposition~\ref{Proposition positive entropy box}, $h(F)>0$, and, in fact, $h(F)=\log2$.
\end{example}

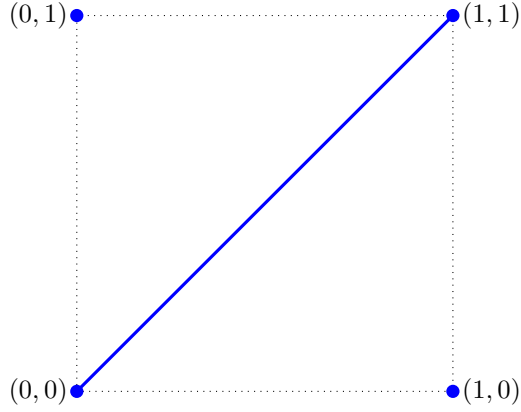
\begin{figure}
\begin{center}
\begin{tikzpicture}[scale=5]
\draw[dotted] (0,0) node[left]{$(0,0)$} -- (0,1) node[left]{$(0,1)$} -- (1,1) node[right] {$(1,1)$} -- (1,0) node[right]{$(1,0)$} -- (0,0);
 \draw (0,1) node[circle, fill=blue, inner sep=0, minimum size=5pt]{};
 \draw (1,0) node[circle, fill=blue, inner sep=0, minimum size=5pt]{};
 \draw (0,0) node[circle, fill=blue, inner sep=0, minimum size=5pt]{};
 \draw (1,1) node[circle, fill=blue, inner sep=0, minimum size=5pt]{};
\draw[very thick, blue](0,0) -- (1,1);
\end{tikzpicture}
\caption{Set-valued function from Example~\ref{Example positive entropy on a nowhere dense set}.}\label{Figure positive entropy on a nowhere dense set}
\end{center}
\end{figure}

One thing which makes Example~\ref{Example positive entropy on a nowhere dense set} interesting is the fact that the positive entropy is really only taking place over the nowhere dense set $\{0,1\}$. Our next two results illustrate that such a thing cannot happen in the context of mappings, or even with continuous set-valued functions.

The following proposition can be found within the proof of a theorem due to Jaquette \cite{Jaquette-Existence_of_top_entropy_preserving_subsystems}. We state the result in a slightly different way than how it appears in \cite{Jaquette-Existence_of_top_entropy_preserving_subsystems}, so we include a proof.

In Proposition~\ref{Proposition dense subset} and Theorem~\ref{Theorem continuous set-valued dense subset}, we will use the following notation.

If $(X,F)$ is a topological dynamical system, and $Z\se X$, then for each $n\in\N$ and $\e>0$, we define $s_{n,\e}(Z,F)$ to be the largest cardinality of an $\e$-separated subset of $\orb_n(Z,F)=\{\xx\in\orb_n(X,F):x_0\in Z\}$.

\begin{proposition}\label{Proposition dense subset}
	Let $X$ be a compact metric space, and let $f:X\rightarrow X$ be continuous. If $Z$ is a dense subset of $X$, then \[h(f)=\lim_{\e\rightarrow0}\limsup_{n\rightarrow\infty}\frac{1}{n}\log s_{n,\e}(Z,f).\]
\end{proposition}

\begin{proof}
	By definition, 
	\[h(f)=\lim_{\e\rightarrow0}\limsup_{n\rightarrow\infty}\frac{1}{n}\log s_{n,\e}(X,f),\]
	so it suffices to show that for each $n\in\N$ and $\e>0$, \[s_{n,\e}(Z,f)\leq s_{n,\e}(X,f)\leq s_{n,\e/2}(Z,f).\]
	Since $Z\se X$, it follows that $s_{n,\e}(Z,f)\leq s_{n,\e}(X,f)$. It remains to show the other inequality.
	
	Recall that $\orb_n(X,f)$ has the metric $D$ defined by $D(\xx,\yy)=\max\{d(x_i,y_i):0\leq i\leq n-1\}$ for $\xx,\yy\in\orb_n(X,f)$. Since $f$ is continuous, the projection map $\pi_0:\orb_n(X,f)\rightarrow X$ is a homeomorphism. Thus, since $Z$ is dense in $X$, it follows that $\orb_n(Z,f)$ is dense in $\orb_n(X,f)$.
	
	Let $n\in\N$ and $\e>0$, and let $S\se\orb_n(X,f)$ be an $(n,\e)$-separated set of maximal cardinality for $f$. Since $\orb_n(Z,f)$ is dense in $\orb_n(X,f)$, for each $\xx\in S$, we may choose $\widetilde{\xx}\in\orb_n(Z,f)$ such that $D(\xx,\widetilde{\xx})<\e/4$. Let $\widetilde{S}=\{\widetilde{\xx}:\xx\in S\}$.
	
	Then, for each $\xx,\yy\in S$ with $\xx\neq\yy$, we have that 
	\begin{eqnarray}
	D\left(\widetilde{\xx},\widetilde{\yy}\right) &\geq& D(\xx,\yy)-D\left(\xx,\widetilde{\xx}\right)-D\left(\yy,\widetilde{\yy}\right)\nonumber\\
	&>&\e-\frac{\e}{4}-\frac{\e}{4}\nonumber\\
	&=&\frac{\e}{2}.\nonumber
	\end{eqnarray}
	It follows that $|S|=|\widetilde{S}|$ and that $\widetilde{S}$ is an $(n,\e/2)$-separated set for $f$. Moreover, since $\widetilde{S}\se\orb_n(Z,f)$, we have that \[s_{n,\e}(X,f)=|S|=|\widetilde{S}|\leq s_{n,\e/2}(Z,f),\] and the result follows.
\end{proof}

Example~\ref{Example positive entropy on a nowhere dense set} illustrates that this result does not hold in general for upper semi-continuous set-valued functions. However, we show in Theorem~\ref{Theorem continuous set-valued dense subset} that it does hold for set-valued functions which are continuous with respect to the Hausdorff metric which we define now.

\begin{definition}
	Let $X$ be a compact metric space with metric $d$. Given a point $x\in X$ and $\e>0$, let $B(x,\e)$ represent the ball of radius $\e$ centered at $x$. We define the \emph{Hausdorff metric}, $\H_d$, on $2^X$ as follows: if $C,D\in2^X$, \[\H_d(C,D)=\sup\left\{\e>0:D\se\bigcup_{c\in C}B(c,\e),\text{ and }C\se\bigcup_{d\in D}B(d,\e)\right\}\]
\end{definition}

\begin{theorem}\label{Theorem continuous set-valued dense subset}
	Let $(X,F)$ be a topological dynamical system such that $F:X\rightarrow 2^X$ is continuous with respect to the Hausdorff metric on $2^X$. If $Z$ is a dense subset of $X$, then \[h(F)=\lim_{\e\rightarrow0}\limsup_{n\rightarrow\infty}\frac{1}{n}\log s_{n,\e}(Z,F).\]
\end{theorem}
\begin{proof}
	From Theorem~\ref{Theorem shift map forward orbits}, we have that the entropy of $F$ is equal to the entropy of the shift map $\sigma$ on $\forb(X,f)$. Thus, since $\sigma$ is a mapping, in light of Proposition~\ref{Proposition dense subset}, it suffices to show that $\forb(Z,f)$ is dense in $\forb(X,f)$.
	
	Recall that $\prod_{i=0}^\infty X$ has the metric $\rho$ defined for $\xx,\yy\in\forb(X,F)$ by \[\rho(\xx,\yy)=\sup_{i\geq0}\frac{d\left(x_i,y_i\right)}{i+1}.\]
	
	Define $\widehat{F}:X\rightarrow2^{\prod X}$ by $\widehat{F}(x)=\forb(x,F)$. Then, $\widehat{F}$ is continuous with respect to the Hausdorff metric $\H_\rho$ on $2^{\prod X}$. Thus, for any $\xx\in\forb(X,F)$ and $\e>0$, we may choose $\delta>0$ to witness the continuity of $\widehat{F}$ at $x_0$. Since $Z$ is dense in $X$, there exists $t\in Z$ such that $d(x_0,t)<\delta$. Then \[\H_\rho\left[\forb(x_0,F),\forb(t,F)\right]<\e,\] so there exists $\yy\in\forb(t,F)\se\forb(Z,F)$ such that $\rho(\xx,\yy)<\e$.
\end{proof}

For mappings on the interval $[0,1]$ we have the following two results concerning periodicity.

\begin{theorem}[{\v{S}}arkovs$'$ki{\u\i} {\cite{Sharkovsky}}]
Define the relation $\prec$ on $\N$ by \[3\prec 5\prec 7\prec\cdots\prec 3\cdot 2\prec 5\cdot 2\prec\cdots\prec 2^3\prec 2^2\prec 2\prec 1.\] If $f:[0,1]\rightarrow[0,1]$ is continuous, and has a periodic point of period $n\in\N$, then it has a periodic point of period $m$, for all $n\prec m$.
\end{theorem}

We also have the following result which relates periodicity to positive topological entropy. A proof may be found in \cite[Section~15.3]{Katok_Hasselblatt}

\begin{theorem}
Let $f:[0,1]\rightarrow[0,1]$ be continuous. Then $h(f)=0$ if, and only if, the period of every periodic point is a power of $2$.
\end{theorem}

The following example illustrates that neither of these results necessarily hold for set-valued functions on the interval.

\begin{example}\label{Example period three zero entropy}
Let $F:[0,1]\rightarrow2^{[0,1]}$ be defined by $F(x)=\{0\}$ for all $x\neq 1/3,2/3,1$, $F(1/3)=\{0,2/3\}$, $F(2/3)=\{0,1\}$, and $F(1)=\{0,1/3\}$. Then $F$ has three periodic orbits of period three and a fixed point but no other periodic orbits. Moreover, $h(F)=0$.
\end{example}

\section{Infinite Topological Entropy and the Structure of Orbit Spaces}\label{Section infinite entropy}

Finally, we explore the concept of infinite topological entropy and its relationship to the structure of the orbit spaces. We begin by presenting sufficient conditions for a set-valued function to have infinite topological entropy. We then consider set-valued functions on $[0,1]$ for which the image and inverse image of a point is connected. We present in Example~\ref{Example zero entropy Hilbert cube} such a function whose entropy is zero, yet whose forward orbit space contains a Hilbert cube (a countable product of non-degenerate closed intervals).

\begin{theorem}\label{Theorem infinite entropy}
Let $(X,F)$ be a topological dynamical system. If there exists an infinite set $A\se X$ such that for all $a\in A$, $F(a)\supseteq A$, then $h(F)=\infty$.
\end{theorem}
\begin{proof}
For each $\e>0$, choose $A_\e$ to be an $\e$-separated subset of $A$ of maximum cardinality, and let $\alpha(\e)=|A_\e|$. Since for each $a\in A$, $A\se F(a)$, we have that for each $n\in\N$, $A^n\se\orb_n(X,F)$. In particular, $A_\e^n$ is a subset of $\orb_n(X,F)$ and is $\e$-separated. Therefore, $s_{n,\e}\geq[\alpha(\e)]^n$ which implies that $h(F,\e)\geq\log\alpha(\e)$.

Since $A$ is an infinite set, $\alpha(\e)\rightarrow\infty$ as $\e\rightarrow0$, so $h(F)=\infty$.
\end{proof}

\begin{corollary}\label{Corollary infinite entropy}
Let $(X,F)$ be a topological dynamical system. If there exists an infinite set $A\se X$ and a $k\in\N$ such that for all $a\in A$, $F^k(a)\supseteq A$, then $h(F)=\infty$.
\end{corollary}
\begin{proof}
By Theorem~\ref{Theorem infinite entropy}, we have that $h(F^k)=\infty$, so from Corollary~\ref{Corollary iterates}, it follows that $h(F)=\infty$.
\end{proof}

For a set-valued function satisfying the hypotheses of either Theorem~\ref{Theorem infinite entropy} or Corollary~\ref{Corollary infinite entropy}, 
its forward orbit space would contain a copy of $A^\N$. It is crucial however that this is a countable product of one infinite set. We demonstrate in Example~\ref{Example zero entropy Hilbert cube} that an orbit space may contain a countable product of infinite sets while the set-valued function has zero entropy.

Before Example~\ref{Example zero entropy Hilbert cube} we define what is meant by a monotone set-valued function.

\begin{definition}
A function $F:X\rightarrow2^X$ is called \emph{monotone} if for each $x\in X$ , $F(x)$ and $F^{-1}(x)$ are each connected.
\end{definition}

A compact, connected, metric space is called a \emph{continuum}. A continuum in which every proper subcontinuum is nowhere dense is called \emph{indecomposable}.

Barge and Diamond prove in \cite{Barge_Diamond-Dynamics_of_maps_finite_graphs} that if $f$ is a piece-wise monotone mapping on a finite graph $G$, then $h(f)>0$ if and only if $\orb(G,f)$ contains an indecomposable subcontinuum.  Example~\ref{Example zero entropy Hilbert cube} demonstrates that this does not hold in general for set-valued functions.

\begin{example}\label{Example zero entropy Hilbert cube}
Let $F:[0,1]\rightarrow2^{[0,1]}$ be the monotone function defined for each $x\in[0,1]$ by $F(x)=[0,x]$. Then $\forb([0,1],F)$ contains copies of the Hilbert cube, and $h(F)=0$.
\end{example}

\begin{proof}
First, note that, in particular, $\forb([0,1],F)$ contains the Hilbert cube \[\prod_{i=1}^\infty\left[\frac{1}{2^i},\frac{1}{2^{i-1}}\right].\]

To show that $h(F)=0$, we show that $h(\sigma)=0$ where $\sigma$ is the shift map on $\forb([0,1],F)$.  First, we claim that the set of non-wandering points is equal to the set of constant sequences (i.e. fixed points for $\sigma$). To see this, let $\xx\in\forb([0,1],F)$, and suppose that $\xx$ is not fixed by $\sigma$. Then there exists some $j\in\N$, such that $x_{j+1}\neq x_j$. From the definition of $F$, it follows that $x_{j+1}<x_j$, and, for all $i>j$, $x_i\leq x_{j+1}<x_j$.

Fix disjoint intervals $I_1$ and $I_2$ such that $x_{j}\in I_1$ and $x_{j+1}\in I_2$, and let $U=\pi_j^{-1}(I_1)\cap\pi_{j+1}^{-1}(I_2)$. Then $\sigma^i(U)$ is disjoint from $U$ for all $i\in\N$. Hence, the only non-wandering points are the fixed points, so $\sigma$ restricted the non-wandering points is the identity. Thus, by Theorem~\ref{Theorem Bowen non-wandering}, $h(\sigma)=0$.
\end{proof}

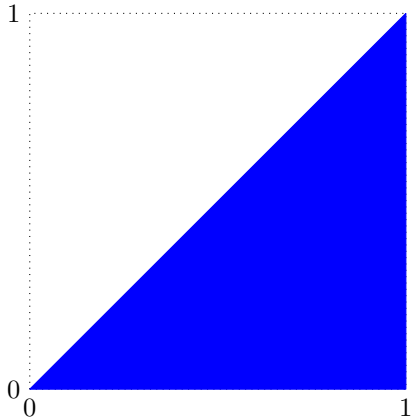
\begin{figure}
\begin{tikzpicture}[scale=5]
\draw[dotted] (0,0) node[left]{0} -- (0,1) node[left]{1} -- (1,1) -- (1,0) node[below]{1} -- (0,0) node[below]{0};
\draw[join=bevel, draw=blue, fill=blue] (1/2,1/2)--(1,1)--(1,0)--(0,0)--(1/2,1/2);
\end{tikzpicture}\caption{Set-valued Function from Example~\ref{Example zero entropy Hilbert cube}}\label{Figure zero entropy Hilbert cube}
\end{figure}

We conclude this paper with the following question.

\begin{question}\label{Question monotone function 0 or infinite}
Does there exist a monotone function $F:[0,1]\rightarrow2^{[0,1]}$ such that $0<h(F)<\infty$.
\end{question}

\bibliography{BIBentropy}
\bibliographystyle{amsplain}
\end{document}